\documentclass[12pt]{article}
\usepackage [a4paper, footskip=1cm, headheight = 16pt, top=2cm, bottom=2.5cm,  right=2cm,  left=2cm]{geometry}
\usepackage{amsmath}
\usepackage{latexsym}
\usepackage{algorithm}
\usepackage{algpseudocode}
\usepackage{tikz}

\usepackage{url}
\usepackage{amssymb}
\usepackage{exscale}
\usepackage{graphicx}
\usepackage{makeidx}

\usepackage{color}
\usepackage{pstricks}

\makeatletter
\def\BState{\State\hskip-\ALG@thistlm}
\makeatother

\newtheorem{thm}{Theorem}
\newtheorem{cor}[thm]{Corollary}

\newtheorem{lem}[thm]{Lemma}

\newenvironment{proof}[1][Proof]
{\par\noindent{\bf #1.} }{\hspace*{\fill}\nolinebreak{$\Box$}\bigskip\par}

\title{\bf On some three color Ramsey numbers for paths, cycles, stripes and
stars}

\author{
\Large Farideh Khoeini \\
\small Department of Mathematical Sciences \\[-0.8ex]
\small Isfahan University of Technology \\[-0.8ex]
\small Isfahan, Iran \\[-0.8ex]
\small {\tt f.khoeini@math.iut.ac.ir}
\\
and\\[-2ex]
\\
\Large Tomasz Dzido\\
\small Institute of Informatics \\[-0.8ex] 
\small Faculty of Mathematics, Physics and Informatics \\[-0.8ex]
\small University of Gda\'{n}sk \\[-0.8ex]
\small 80-308 Gda\'{n}sk, Poland\\[-0.8ex]
\small {\tt tdz@inf.ug.edu.pl}\\
}

\begin{document}

\maketitle
\thispagestyle{empty}

\vspace{-0.5cm}
\begin{abstract}
For given graphs $G_{1}, G_{2}, ... , G_{k}, k \geq 2$, the \emph{multicolor Ramsey number} $R(G_{1}, G_{2}, ... , G_{k})$ is the smallest integer $n$ such that if we arbitrarily color the edges of the complete graph of order $n$
with $k$ colors, then it always contains a monochromatic copy of
$G_{i}$ colored with $i$, for some $1 \leq i \leq k$. The bipartite Ramsey number $b(G_1, \cdots, G_k)$ is the least positive integer $b$ such that any coloring of the edges of $K_{b,b}$ with $k$ colors will result in a monochromatic copy of bipartite $G_i$ in the $i$-th color, for some $i$, $1 \le i \le k$.  

\medskip 
There is very little known about $R(G_{1},\ldots, G_{k})$ even for very special graphs, there are a lot of open cases (see \cite{15}). In this paper, by using bipartite Ramsey numbers we obtain the exact values of some multicolor Ramsey numbers. We show that for sufficiently large $n_{0}$ and three following cases:
\begin{itemize}
\item [1.] $n_{1}=2s$, $n_{2}=2m$ and $m-1<2s$,
\item [2.] $n_{1}=n_{2}=2s$,
\item [3.] $n_{1}=2s+1$, $n_{2}=2m$ and $s<m-1<2s+1$, we have
\end{itemize} 
$$R(C_{n_0}, P_{n_{1}},P_{n_{2}}) = n_0 + \Big \lfloor \frac{n_1}{2} \Big \rfloor + \Big \lfloor \frac{n_2}{2} \Big \rfloor -2.$$

\medskip 
We prove that $R(P_n,kK_{2},kK_{2})=n+2k-2$ for large $n$.  In addition, we prove that for even $k$, $R((k-1)K_{2},P_{k},P_{k})=3k-4$. For $s < m-1<2s+1$ and $t\geq m+s-1$, we obtain that $R(tK_{2},P_{2s+1},P_{2m})=s+m+2t-2$
where $P_{k}$ is a path on $k$ vertices and $tK_{2}$ is a matching of size $t$.

\medskip
We also provide some new exact values or generalize known results for other multicolor Ramsey numbers of paths, cycles, stripes and stars versus other graphs. 
\end{abstract}



\bigskip
\section{Introduction}

All graphs in this paper are undirected, finite and simple. The union of two graphs $G$ and $H$, denoted by $G\cup H,$ is a graph with vertex set $V(G\cup H)$ and edge set $E(G)\cup E(H)$. The join of two graphs $G$ and $H$, denoted by $G+H,$ is a graph with vertex set $V(G\cup H)$ and edge set $E(G\cup H) \cup \lbrace uv \vert u\in G, v\in H\rbrace$.  The union of $k$ disjoint copies of the same graph $G$ is denoted by $kG$. $\overline{G}$ stands for the complement of the graph $G$.  We denote by $G[U]$ the subgraph of $G$ induced by the vertex set $U$. By $P_n$ and $C_n$ we denote the path and cycle on $n$ vertices. For a $3$-edge coloring (say blue, red and green) of a graph $G$, we denote by $G^{b}$ (resp. $G^{r}$ and $G^{g}$) the induced subgraph by the edges of color blue (resp. red and green). 

\medskip
For given graphs $G_{1}, G_{2}, ... , G_{k}, k \geq 2$, the \emph{multicolor Ramsey number} $R(G_{1}, G_{2}, ... , G_{k})$ is the smallest integer $n$ such that if we arbitrarily color the edges of the complete graph of order $n$ with $k$ colors, then it always contains a monochromatic copy of
$G_{i}$ colored with $i$, for some $1 \leq i \leq k$.  The existence of such a positive integer is guaranteed by Ramsey's classical result \cite{1}. 
The bipartite Ramsey number $b(G_1, \cdots, G_k)$ is the least positive integer $b$ such that any coloring of the edges of $K_{b,b}$ with $k$ colors will result in a monochromatic copy of bipartite $G_i$ in the $i$-th color, for some $i$, $1 \le i \le k$. 

\medskip
There is very little known about $R(G_{1},\ldots, G_{k})$ even for very special graphs, there are a lot of open cases (see \cite{15}). In this paper, by using bipartite Ramsey numbers we obtain a proof of the exact values of some three-color Ramsey numbers. 

\medskip
In 2006 Dzido \emph{et al.} \cite{24} proved that $R(P_4,P_4,C_n)=n+2$ and $R(P_3,P_5,C_n)=n+1$. In subsection 3.1 we generalize their results  
and show that for sufficiently large $n_{0}$ and three following cases:
\begin{itemize}
\item [1.] $n_{1}=2s$, $n_{2}=2m$ and $m-1<2s$,
\item [2.] $n_{1}=n_{2}=2s$,
\item [3.] $n_{1}=2s+1$, $n_{2}=2m$ and $s<m-1<2s+1$, we have
\end{itemize} 
$$R(C_{n_0}, P_{n_{1}},P_{n_{2}}) = n_0 + \Big \lfloor \frac{n_1}{2} \Big \rfloor + \Big \lfloor \frac{n_2}{2} \Big \rfloor -2.$$

\noindent For the Ramsey number of paths a well-known theorem of Gerencs\'{e}r and Gy\'{a}rf\'{a}s \cite{17} states that $R(P_{n},P_{m})=m+\lfloor n/2 \rfloor -1$ where $m \geq n \geq  2$. In 1975 \cite{6} these authors determined $R(P_{n_1},P_{n_2},P_{n_3})$ for the case $n_{1}\geq 6(n_{2}+n_{3})^{2}$
and they conjectured that 
\begin{equation*}
R(P_{n},P_{n},P_{n}) = \left\{
\begin{array}{rl}
2n-2 & \text{if } n \text{ is even,}\\
2n-1 & \text{if } n \text{ is odd.}
\end{array} \right.
\end{equation*}
In 2007 this conjecture was established by Gy\'{a}rf\'{a}s \emph{et al.}  \cite{18} for sufficiently large $n$. We can apply the result for $R(C_{n_0}, P_{n_{1}},P_{n_{2}})$ to $P_{n_0}$ instead of $C_{n_0}$ to obtain the same result as in \cite{6}.   

\medskip
The Ramsey number of star versus path was completely determined by Parsons \cite{2}. In subsections 3.3 and 3.4 we investigate multicolor Ramsey number of a cycle $C_n$ or a path $P_n$ versus stars and strips for large value of $n$. 

\medskip
In \cite{5} Maherani \emph{et al.} proved that $R(P_{3}, kK_{2}, nK_{2})=2k+n-1$  for $k\geq n\geq 3$.  In this paper we show that $R(P_n,kK_{2},kK_{2})=n+2k-2$ for large $n$.  In addition we prove that for even $k$, $R((k-1)K_{2},P_{k},P_{k})=3k-4$. For $s < m-1<2s+1$ and $t\geq m+s-1$, we obtain that $R(tK_{2},P_{2s+1},P_{2m})=s+m+2t-2.$

\medskip
We also provide some new exact values or generalize known results for other multicolor Ramsey numbers of paths, cycles, stripes and stars versus other graphs.

\bigskip
\section{Main results}

\begin{thm}\label{SS}
For every graph $H$ and bipartite graphs $G_1, \ldots, G_k$, we have
\[R(H,G_1,\ldots,G_k)\leq R(H, K_{b,b}),\]
where $b=b(G_1, \ldots, G_k)$. 
\end{thm}
\begin{proof}
Assume $R(H, K_{b,b})=n$, we will show that for any coloring of the edges of the complete graph $K_{n}$ by $k+1$ colors there exists a color $i$ for which the corresponding color class contains $G_{i}$ as a subgraph. 
 
Suppose that $G=K_{n}$ is $k+1$-edge colored such that $G$ does not contain $H$ of color $1$.  We show that there is a copy of $G_{i}$ of color $i$ in $G$ for some $2\leq i\leq k+1$. Now we merge $k$ colors classes $2, \ldots, k+1$. Suppose that new class is black. Since $R(H, K_{b,b})=n$, we have a black copy of $K_{b,b}$. By coming back to its original $1, \ldots, k$ coloring, we get that there exists a complete bipartite subgraph $H=K_{b,b}$ where its edges are colored with $2,\ldots,k+1$ (observe that there is no edges of color $1$ in $H$). Thus we use $b=b(G_1, \ldots, G_k)$  which guarantees that $K_{b,b}$ contains a copy of $G_{i}$ of color $i$ for some $2\leq i\leq k+1$.
\end{proof} 

The following theorems appear in \cite{3} and \cite{30}, respectively.
\begin{thm}[\cite{3}]\label{o}
$R(P_{m},K_{n,k})\leq k+n+m-2.$
\end{thm}
\begin{thm}[\cite{30}]\label{Z}
$R(tK_{2},K_{n,n})=max \lbrace n+2t-1,2n+t-1 \rbrace.$
\end{thm}

It is easy to verify the following theorems by the previous Theorems \ref{SS},  \ref{o} and \ref{Z}.
\begin{thm}\label{i}
For bipartite graphs $G_1, \ldots, G_k$, we have
$$R(P_{m}, G_1,\ldots,G_k)\leq 2b+m-2,$$ 
where $b=b(G_1,\ldots,G_k)$. 
\end{thm}

\begin{thm}\label{c}
For bipartite graphs $G_1, \ldots, G_k$, we have
\[
R(tK_{2}, G_1,\ldots, G_k) \leq
\begin{cases}
2b+t-1  & \text{if } t \leq b , \\
b+2t-1 & \text{if } t \geq b.
\end{cases}
\]
where $b=b(G_1, \ldots, G_k)$. 
\end{thm}

It is easy to see that for bipartite graphs $G_1, G_2, \ldots, G_k$ we have $R(G_1, G_2, \ldots, G_k) \leq  2b(G_1, G_2, \ldots, G_k)$.

\begin{cor}\label{AA}
If $R(G_1, G_2, \ldots, G_k) =  2b(G_1, G_2, \ldots, G_k)=2t$ for bipartite graphs $G_1, \ldots, G_k$, then 
\[
R(tK_{2},G_1,\ldots,G_k)=3t-1.
\]
\end{cor}
\begin{proof}
By Theorem \ref{c}, the upper bound is clear. To see the lower bound consider the graph $G=K_{3t-2}=K_{2t-1}+ K_{t-1}$.
Since $R(G_1, \ldots, G_k)=2t$, we take a $k$-coloring of $E(K_{2t-1})$ which does not contain a copy of $G_{i}$ in color $i$ for each $1 \leq i\leq k$. The remaining edges of $G$ we color with the last $k+1$-th color. Thus $G^{k+1}$ contains no copy of $tK_{2}$. This proves the corollary.
\end{proof}

\bigskip
\section{Corollaries}

\subsection{$R(C_{n_{0}},P_{n_{1}},P_{n_{2}})$ for large $n_{0}$}

In this section, we determine the value of $R(C_{n_{0}},P_{n_{1}},P_{n_{2}})$ for large $n_{0}$ and special cases of $n_{1}, n_{2}$. In order to prove it, we first recall the result of Bondy and Erd\H{o}s. In 1973 \cite{7} they proved that for $n> n_{1}(r)$ (that is for sufficiently large $n$), $R(C_{n},K_{r,r})=n+r-1$. More precisely, they showed the following.

\begin{thm}[\cite{7}] \label{M}
 For $n> n_{1}(r,t)$, $R(C_{n},K^{t+1}_{r})=t(n-1)+r$ such that $K^{t+1}_{r}$ is the complete $(t+1)$-partite graph with parts of 
size $r$.
\end{thm}

\begin{thm}\label{q}For sufficiently large $n_{0}$ and three following cases:
\begin{itemize}
\item [1.] $n_{1}=2s$, $n_{2}=2m$ and $m-1<2s$,
\item [2.] $n_{1}=n_{2}=2s$,
\item [3.] $n_{1}=2s+1$, $n_{2}=2m$ and $s<m-1<2s+1$, we have
\end{itemize} 
$$R(C_{n_0}, P_{n_{1}},P_{n_{2}}) = n_0 + \Big \lfloor \frac{n_1}{2} \Big \rfloor + \Big \lfloor \frac{n_2}{2} \Big \rfloor -2.$$
\end{thm}
\begin{proof}
For the upper bound, by  Theorem \ref{SS}, $R(C_{n_{0}},P_{n_{1}},P_{n_{2}})\leq R(C_{n_{0}},K_{b,b})$ such that $b=b(P_{n_{1}},P_{n_{2}})$. By Theorem \ref{M}, $R(C_{n_{0}},P_{n_{1}},P_{n_{2}})\leq n_{0}+b-1$.
On the other hand, for all three cases $b=\Big \lfloor \frac{n_1}{2} \Big \rfloor + \Big \lfloor \frac{n_2}{2} \Big \rfloor -1$ (see Theorem \ref{t}) and $R(C_{n_{0}},P_{n_{1}},P_{n_{2}}) \leq n_{0}+\Big \lfloor \frac{n_1}{2} \Big \rfloor + \Big \lfloor \frac{n_2}{2} \Big \rfloor -2$.

For the lower bound, consider the graph $G=K_{R-1}\cup K_{\frac{n_2}{2}-1}$ such that $R=R(C_{n_{0}},P_{n_{1}})$. It is known that $R=n_{0}+\Big \lfloor \frac{n_1}{2} \Big \rfloor -1$ for $n_{0}\geq n_{1}\geq 2$. It is clear that there is a two-coloring (blue-red) of $K_{R-1}$ such that $G^{b}$ contains no copy of $C_{n_{0}}$ and $G^{r}$ contains no copy of $P_{n_{1}}$. Next color the remaining subgraph $K_{\frac{n_2}{2}-1}$ with red. Since $\frac{n_2}{2}-1 < n_{1}$, there is no a red copy of $P_{n_{1}}$ in $K_{\frac{n_2}{2}-1}$ as well. Consider $\overline{G}=\overline{K}_{R-1}+ \overline{K}_{\frac{n_2}{2}-1}$ and color it with green. Thus $G^{g}$ contains no copy of $P_{n_{2}}$. 
The equality follows.
\end{proof}

In 1975 Faudree and Schelp \cite{6} proved that if $n_{0}\geq 6(n_{1}+n_{2})^{2}$, then $R(P_{n_{0}},P_{n_{1}},P_{n_{2}})=n_{0}+\lfloor n_{1}/2 \rfloor+\lfloor n_{2}/2 \rfloor-2$ for $n_{1}, n_{2}\geq 2$.  
Since $R(P_{n_{0}},P_{n_{1}},P_{n_{2}}) \leq R(C_{n_{0}},P_{n_{1}},P_{n_{2}})$, we can apply Theorem \ref{q} to $P_{n_{0}}$ instead of $C_{n_{0}}$ to obtain the same results as in \cite{6}. 

\begin{cor}
For sufficiently large $n_{0}$ and three following cases:
\begin{itemize}
\item [1.] $n_{1}=2s$, $n_{2}=2m$ and $m-1<2s$,
\item [2.] $n_{1}=n_{2}=2s$,
\item [3.] $n_{1}=2s+1$, $n_{2}=2m$ and $s<m-1<2s+1$, we have
\end{itemize} 

$$R(P_{n_0}, P_{n_{1}},P_{n_{2}}) = n_0 + \Big \lfloor \frac{n_1}{2} \Big \rfloor + \Big \lfloor \frac{n_2}{2} \Big \rfloor -2.$$
\end{cor}

\bigskip
\subsection{$R(tK_{2},P_{k},P_{k^{'}})$}

In 1975 Faudree and Schelp \cite{9} determined $b(P_{n},P_{k})$ for all $n$ and $k$.
\begin{thm}[\cite{9}] \label{t}
For $s, m \in Z$,
\begin{itemize}
\item[1.] $b(P_{2s},P_{2m})=(s+m-1,s+m-1)$,
\item[2.] $b(P_{2s+1},P_{2m})=(s+m,s+m-1)$ for $s\geq m-1$,
\item[3.] $b(P_{2s+1},P_{2m})=(s+m-1,s+m-1)$ for $s < m-1$,
\item[4.] $b(P_{2s+1},P_{2m+1})=(s+m,s+m-1)$ for $s \neq m$,
\item[5.] $b(P_{2s+1},P_{2s+1})=(2s+1,2s-1)$.
\end{itemize}
\end{thm}

\begin{thm}\label{p}For positive integers $k, m, s, t,$
\begin{itemize}
\item[(i)] For even $k$, $R((k-1)K_{2},P_{k},P_{k})=3k-4$.
\item[(ii)] For $s < m-1<2s+1$ and $t\geq m+s-1$, $R(tK_{2},P_{2s+1},P_{2m})=s+m+2t-2.$
\end{itemize}
\end{thm}
\begin{proof}(i)
Theorem \ref{c}  implies that $R(tK_{2},P_{k},P_{k}) \leq 2b+t-1$ for $t\leq b$ and $b=b(P_{k},P_{k})$. Theorem \ref{t} (that is $b=b(P_{k},P_{k})=k-1$ for even $k$)  completes the proof for the upper bound.

\noindent It is known that $R=R((k-1)K_{2},P_{k})=2k+\lfloor k/2 \rfloor -3$ (see \cite{30}). Consider the graph $G=K_{R-1}\cup K_{k/2 -1}$. Clearly $K_{R-1}$ has a two coloring (blue-red) such that $K_{R-1}^{b}$ contains no copy of  $(k-1)K_{2}$ and $K_{R-1}^{r}$ contains no copy of $P_{k}$. We take this coloring and  next we color the edges of $K_{k/2 -1}$ with red. The  edges of $\overline{G}$ we color with green. Hence $G^{g}$ contains no copy of $P_{k}$ as well. This gives the desired lower bound.

(ii) As before, we have $R(tK_{2},P_{2s+1},P_{2m}) \leq b+2t-1 \leq s+m+2t-2$ for $s < m-1$ and $t\geq m+s-1$. Now, let  $G=K_{R-1}\cup K_{m-1}$ such that $R=R(tK_{2}, P_{2s+1})=2t+\lfloor (2s+1)/2 \rfloor-1$ for $t> \lfloor (2s+1)/2 \rfloor$ (see \cite{30}). Clearly $K_{R-1}$ can be colored in such a way that $K_{R-1}^{b}$ contains no copy of $tK_{2}$ and $K_{R-1}^{r}$  contains no copy of $P_{2s+1}$. Next we color the subgraph $K_{m-1}$ with red and the edges of $\overline{G}$ with green. Then $G^{g}$ contains no copy of $P_{2m}$ and the proof is complete.
\end{proof}

\bigskip
\subsection{$R(C_n(P_n),K_{1,k_{1}}, K_{1,k_{2}}, \ldots, K_{1,k_{t}}, m_{1}K_{2}, m_{2}K_{2}, \ldots, m_{s}K_{2})$ for large $n$}

\begin{thm}\label{d}
Let $m_{1}, m_{2},\ldots,m_{s}$ and $k_{1}, k_{2}, \ldots, k_{t}$ be positive integers and  $n> n_{1}(b)$ where 
$b=b(K_{1,k_{1}}, K_{1,k_{2}}, \ldots, K_{1,k_{t}}, m_{1}K_{2}, m_{2}K_{2}, \ldots, m_{s}K_{2})$, then
\[
R(C_n,K_{1,k_{1}} ,K_{1,k_{2}}, \ldots,K_{1,k_{t}}, m_{1}K_{2}, m_{2}K_{2}, \ldots, m_{s}K_{2})\leq n+b-1.
\]
\end{thm}
\begin{proof}
By using the same argument as in Theorem \ref{SS}, we have\\ $R(C_n,K_{1,k_{1}}, K_{1,k_{2}}, \ldots, K_{1,k_{t}}, m_{1}K_{2}, m_{2}K_{2}, \ldots, m_{s}K_{2}) \leq R(C_n,K_{b,b})$ \\ where $b=b(K_{1,k_{1}}, K_{1,k_{2}}, \ldots, K_{1,k_{t}}, m_{1}K_{2}, m_{2}K_{2}, \ldots, m_{s}K_{2})$. Next we apply Theorem \ref{M}.
\end{proof}

\begin{lem}[\cite{8}]\label{f} Let $m_{1}, m_{2},\ldots,m_{s}$ and $k_{1}, k_{2}, \ldots, k_{t}$ be positive integers with $\Lambda=\sum_{i=1}^{s}(m_{i}-1)$ and $\sum=\sum_{i=1}^{t}(k_{i}-1)$. Then $b(K_{1,k_{1}}, K_{1,k_{2}}, \ldots, K_{1,k_{t}}, m_{1}K_{2}, m_{2}K_{2}, \ldots, m_{s}K_{2})=b$, where
\[
b =
\begin{cases}
\Lambda+1  & \text{if } \Sigma<\lfloor (\Lambda+1)/2 \rfloor , \\
\Sigma+\lfloor \Lambda/2 \rfloor+1 & \text{if }  \Sigma \geq \lfloor (\Lambda+1)/2 \rfloor .
\end{cases}
\]
\end{lem}

Combining Theorem \ref{d} and Lemma \ref{f}, we obtain the following.
\begin{thm}\label{Far}
Let $m_{1}, m_{2},\ldots,m_{s}$ and $k_{1}, k_{2}, \ldots, k_{t}$ be positive integers with $\Lambda=\sum_{i=1}^{s}(m_{i}-1)$ and $\sum=\sum_{i=1}^{t}(k_{i}-1)$. If $\Sigma \leq \lfloor (\Lambda+1)/2 \rfloor$ 
for $n> n_{1}(\Lambda)$,  then 
\[
R(C_n,K_{1,k_{1}},K_{1,k_{2}},\ldots,K_{1,k_{t}},m_{1}K_{2},m_{2}K_{2},\ldots,m_{s}K_{2})= n+\Lambda.
\]
\end{thm}
\begin{proof}
For the upper bound, we apply Theorem \ref{d} and Lemma \ref{f}.

\noindent For the lower bound color all edges of $G=K_{n-1}\cup \overline{K}_{\Lambda}$ by color $1$ and for all edges of $\overline{G}=\overline{K}_{n-1}+ K_{\Lambda}$ consider the following coloring.
Color $\overline{K}_{n-1}+ K_{m_{1}-1}$ and $\overline{K}_{m_{1}-1}+\overline{K}_{\sum_{i=2}^{s}(m_{i}-1)}$ by color $2$ 
and color $\overline{K}_{n-1}+ K_{m_{2}-1}$ and $\overline{K}_{m_{2}-1}+\overline{K}_{\sum_{i=3}^{s}(m_{i}-1)}$ by color $3$ and
in the general color $\overline{K}_{n-1}+ K_{m_{j}-1}$ and $\overline{K}_{m_{j}-1}+\overline{K}_{\sum_{i=j+1}^{s}(m_{i}-1)}$ by color $j+1$ and finally color $\overline{K}_{n-1}+ K_{m_{s}-1}$ by color $s$.
Then $G^{1}$ contains no copy of $C_{n}$, $G^{i}$ contains no copy of $m_{i}K_{2}$ for $1\leq i \leq s$, as desired.\\
\end{proof}

\begin{cor}\label{s}For given positive integers $m_{1}, m_{2},\ldots,m_{s}$ and $k_{1},k_{2},\ldots,k_{t}$ with $\Lambda=\sum_{i=1}^{s}(m_{i}-1)$, $\sum=\sum_{i=1}^{t}(k_{i}-1)$, sufficiently large $n$ and  $\Sigma \leq \lfloor (\Lambda+1)/2 \rfloor$, 
\[
R(P_{n},K_{1,k_{1}},K_{1,k_{2}},\ldots,K_{1,k_{t}},m_{1}K_{2},m_{2}K_{2},\ldots,m_{s}K_{2}) =n+\Lambda.
\]
\end{cor}

\medskip
\subsection{$R(C_n,kK_{2},kK_{2})$ and $R(P_n,kK_{2},kK_{2})$ for large $n$}

\begin{lem}[\cite{4}]\label{h}
$b(mK_{2}, nK_{2})=m+n-1$.
\end{lem}

\begin{thm}\label{h1}
$R(C_n,kK_{2},kK_{2})= n+2k-2$, for sufficiently large $n$.
\end{thm}

\begin{proof}
Theorems \ref{SS},\ref{M} and Lemma \ref{h} give us the desired upper bound.

In \cite{30} it is shown that $R(C_{n}, kK_{2})=n+k-1$ for $k \leq \lfloor \frac{n}{2} \rfloor$. Consider the graph $G=K_{n+k-2} \cup \overline{K}_{k-1}$. There is two-coloring of $K_{n+k-2}$ such that $G^{b}$ contains no copy of $C_{n}$ and $G^{r}$ contains no copy of $kK_{2}$. Next color $\overline{G}=\overline{K}_{n+k-2}+K_{k-1}$ with color green. The
theorem follows.
\end{proof}

In \cite{5} Maherani \emph{et al.} proved that $R(P_{3}, kK_{2}, tK_{2})=2k+t-1$  for $k\geq t\geq 3$. 

\begin{thm}
$R(P_n,kK_{2},kK_{2})= n+2k-2$,  for sufficiently large $n$.
\end{thm}

\begin{proof}
Clearly $R(P_n,kK_{2},kK_{2})\leq R(C_{n},kK_{2},kK_{2})$. Theorem \ref{h1} gives us the upper bound.

In \cite{30} it is shown that $R(P_{n}, kK_{2})=n+k-1$ for $k \leq \lfloor \frac{n}{2} \rfloor$. We obtain the lower bound by considering the same coloring as in the proof of Theorem \ref{h1}.  
\end{proof}

\bigskip
\subsection{Multicolor Ramsey number for stripe versus path and stars}

\begin{lem}[\cite{8}]\label{RE}
For positive integers $m, k_{1},\ldots,k_{r}\geq 2$ and $\sum=\sum_{i=1}^{r}(k_{i}-1)$, 
\[
b(P_{m},K_{1,k_{1}},\ldots,K_{1,k_{r}})=
\begin{cases}
\Sigma + \frac{m}{2}  & \text{if } \Sigma \geq \frac{m}{2} , \text{$m$  even,}\\
\Sigma + \frac{m+1}{2}& \text{if } \Sigma \geq \frac{m-1}{2}, \text{$m$ odd, } \Sigma \equiv 0 \text{ mod}(\frac{m-1}{2}),\\
\Sigma+ \frac{m-1}{2} & \text{if } \Sigma \geq \frac{m-1}{2}, \text{$m$ odd, } \Sigma \not \equiv 0 \text{ mod}(\frac{m-1}{2}),\\
2\Sigma+1 & \text{if }  \frac{1}{2}\lfloor \frac{m}{2} \rfloor +1 \leq \Sigma < \lfloor\frac{m}{2} \rfloor +1,\\
\lfloor \frac{m+1}{2} \rfloor  & \text{if } \Sigma < \frac{1}{2}\lfloor m/2 \rfloor.
\end{cases}
\]
\end{lem}
\begin{thm}For $ \Sigma< \frac{1}{2}\lfloor m/2 \rfloor$, even $m=2s$ and  $2 \leq t\leq s$,
\[
R(tK_{2},P_{m},K_{1,k_{1}},\ldots,K_{1,k_{r}})=m+t-1. 
\]
\end{thm}
\begin{proof}
Suppose that $ \Sigma < \frac{1}{2}\lfloor m/2 \rfloor$,  $m=2s$ is even and  $2 \leq t\leq s$.  By Theorem \ref{c} and Lemma \ref{RE} we have
$R(tK_{2},P_{m},K_{1,k_{1}},\ldots,K_{1,k_{r}})\leq m+t-1 $. To obtain the lower bound, divide the vertex set of $G=K_{m+t-2}$ into two parts $A$ and $B$, where
$|A|=m-1$ and $|B|=t-1$. Next, color the edges of $G[A]$ with the second color and other edges with the first color.
\end{proof}

\bigskip
\subsection{$R(tK_{2},P_{3},C_{2n})$ for $t \leq n$}
\begin{thm}[\cite{15}]
$R(P_{3},C_{2n})=2n$.
\end{thm}
\begin{thm}For positive integers $t \leq n$, $R(tK_{2},P_{3},C_{2n})= 2n+t-1$.
\end{thm}
\begin{proof}
It is easy to see that $b(P_{3},C_{2n})=n$ for $n\geq 2$.
By Theorem \ref{c} we have $R(tK_{2},C_{2n},P_{3})\leq 2n+t-1$ where $t \leq n$.
To see the lower bound, assume that $G=K_{2n-1}\cup \overline{K}_{t-1}$.
It is clear that there is a two coloring (blue-red) of $E(K_{2n-1})$ such that there is no blue copy $C_{2n}$ and no red copy $P_{3}$. We take this coloring and next we color the edges of $\overline{G}=\overline{K}_{2n-1}+ K_{t-1}$ with color green. There is no green copy $tK_{2}$ as well.
\end{proof}

\bigskip
\subsection{$R(tK_{2},C_{2m},C_{4})$}

\begin{thm}[\cite{14}]\label{z}
$b(C_{2m},K_{2,2})=m+1$ for $m\geq 4.$
\end{thm}
\begin{thm}[\cite{30}]
$R(tK_{2},C_{n})=max \lbrace n+2t-1- \lfloor n/2 \rfloor, n+t-1  \rbrace$ for $n\geq 3.$
\end{thm}

\begin{thm}For positive integer $m\geq 4$
\begin{itemize}
\item[(1)]  If $t \geq m+1$, $R(tK_{2},C_{2m},C_{4})=m+2t$.
\item[(2)] If $t \leq m$, $2m+t \leq R(tK_{2},C_{2m},C_{4})\leq 2m+t+1$.
\end{itemize}
\end{thm}
\begin{proof}(1) By Theorems \ref{c} and \ref{z} we have $R(tK_{2},C_{2m},C_{4})\leq b+2t-1=m+2t$ where $b=b(C_{2m},K_{2,2})$ and $t \geq m+1$. To see the lower bound, consider the graph $G=K_{R-1}\cup K_{1}$ so that $R=R(tK_{2},C_{2m})=m+2t-1$ for $t \geq m+1$.
It is clear that there is a two coloring (blue-red) of $E(K_{R-1})$ such that there is no blue copy of $tK_{2}$ and no red copy of $C_{2m}$. We take this coloring. Next we color the edges of $\overline{G}=\overline{K}_{R-1}+ \overline{K}_{1}$ by green. So there is no green copy of $C_{4}$ as well.

\medskip 
(2) By Theorem \ref{c} we have $R(tK_{2},C_{2m},C_{4})\leq 2b+t-1$ where $b=b(C_{2m},K_{2,2})$ and $t\leq m+1$. By Theorem \ref{z} we obtain $R(tK_{2},C_{2m},C_{4})\leq 2m+t+1$. To see the lower bound, consider the graph $G=K_{R-1}\cup K_{1}$ so that $R=R(tK_2,C_{2m})=2m+t-1$.
It is clear that there is a two coloring (blue-red) of $E(K_{R-1})$ such that there is no blue copy of $tK_{2}$ and no red copy $C_{2m}$. We take such a coloring and next we color the edges of $\overline{G}=\overline{K}_{R-1}+ \overline{K}_{1}$ with green. There is no green copy of $C_4$ as the proof is complete.
\end{proof}

\end{document}